\documentclass[11pt]{amsart}

\usepackage{enumerate}
\usepackage{amsmath,amsthm,verbatim,amssymb,amsfonts,amscd,amsopn,amsxtra,graphicx,lmodern}
\usepackage{hyperref}
\usepackage{tikz-cd} 
\usepackage{graphics}
\usepackage{mathrsfs}
\usepackage{relsize}
\usepackage{enumitem}
\usepackage[utf8]{inputenc}
\usepackage{csquotes}
\usepackage{amsthm}
\usepackage{thmtools}
\usepackage{mathtools}
\usepackage{microtype}
\usepackage{pdfrender,xcolor}
\usepackage[sort cites=true, backend=biber]{biblatex}
\usepackage{biblatex}

\tikzset{
	symbol/.style={
		draw=none,
		every to/.append style={
			edge node={node [sloped, allow upside down, auto=false]{$#1$}}}
	}
}

\begin{document}
	\pdfrender{StrokeColor=black,TextRenderingMode=2,LineWidth=0.2pt}	
	
	\title{On the implicit constant fields and key polynomials for valuation algebraic extensions}

	\author{Arpan Dutta}

	\def\NZQ{\mathbb}               
	\def\NN{{\NZQ N}}
	\def\QQ{{\NZQ Q}}
	\def\ZZ{{\NZQ Z}}
	\def\RR{{\NZQ R}}
	\def\CC{{\NZQ C}}
	\def\AA{{\NZQ A}}
	\def\BB{{\NZQ B}}
	\def\PP{{\NZQ P}}
	\def\FF{{\NZQ F}}
	\def\GG{{\NZQ G}}
	\def\HH{{\NZQ H}}
	\def\UU{{\NZQ U}}
	\def\P{\mathcal P}
	
	%
	%
	\let\union=\cup
	\let\sect=\cap
	\let\dirsum=\oplus
	\let\tensor=\otimes
	\let\iso=\cong
	\let\Union=\bigcup
	\let\Sect=\bigcap
	\let\Dirsum=\bigoplus
	\let\Tensor=\bigotimes
	
	\theoremstyle{plain}
	\newtheorem{Theorem}{Theorem}[section]
	\newtheorem{Lemma}[Theorem]{Lemma}
	\newtheorem{Corollary}[Theorem]{Corollary}
	\newtheorem{Proposition}[Theorem]{Proposition}
	\newtheorem{Problem}[Theorem]{}
	\newtheorem{Conjecture}[Theorem]{Conjecture}
	\newtheorem{Question}[Theorem]{Question}
	
	\theoremstyle{definition}
	\newtheorem{Example}[Theorem]{Example}
	\newtheorem{Examples}[Theorem]{Examples}
	\newtheorem{Definition}[Theorem]{Definition}
	
	\theoremstyle{remark}
	\newtheorem{Remark}[Theorem]{Remark}
	\newtheorem{Remarks}[Theorem]{Remarks}
	
	\newcommand{\trdeg}{\mbox{\rm trdeg}\,}
	\newcommand{\rr}{\mbox{\rm rat rk}\,}
	\newcommand{\sep}{\mathrm{sep}}
	\newcommand{\ac}{\mathrm{ac}}
	\newcommand{\ins}{\mathrm{ins}}
	\newcommand{\res}{\mathrm{res}}
	\newcommand{\Gal}{\mathrm{Gal}\,}
	\newcommand{\ch}{\mathrm{char}\,}
	\newcommand{\Aut}{\mathrm{Aut}\,}
	\newcommand{\kras}{\mathrm{kras}\,}
	\newcommand{\dist}{\mathrm{dist}\,}
	\newcommand{\ord}{\mathrm{ord}\,}
	
	\newcommand{\n}{\par\noindent}
	\newcommand{\nn}{\par\vskip2pt\noindent}
	\newcommand{\sn}{\par\smallskip\noindent}
	\newcommand{\mn}{\par\medskip\noindent}
	\newcommand{\bn}{\par\bigskip\noindent}
	\newcommand{\pars}{\par\smallskip}
	\newcommand{\parm}{\par\medskip}
	\newcommand{\parb}{\par\bigskip}
	\let\epsilon\varepsilon
	\let\phi=\varphi
	\let\kappa=\varkappa
	
	\def \a {\alpha}
	\def \b {\beta}
	\def \s {\sigma}
	\def \d {\delta}
	\def \g {\gamma}
	\def \o {\omega}
	\def \l {\lambda}
	\def \th {\theta}
	\def \D {\Delta}
	\def \G {\Gamma}
	\def \O {\Omega}
	\def \L {\Lambda}
	%
	%
	\textwidth=15cm \textheight=22cm \topmargin=0.5cm
	\oddsidemargin=0.5cm \evensidemargin=0.5cm \pagestyle{plain}

	\address{Department of Mathematics, IISER Mohali,
		Knowledge City, Sector 81, Manauli PO,
		SAS Nagar, Punjab, India, 140306.}
	\email{arpan.cmi@gmail.com}
	
	\date{\today}
	
	\thanks{This work was supported by the Post-Doctoral Fellowship of the National Board of Higher Mathematics, India.}
	
	
	
	\maketitle
	
\begin{abstract}
	This article is a natural continuation of our previous works [\ref{Dutta min fields implicit const fields}] and [\ref{Dutta min pairs inertia ram deg impl const fields}]. In this article, we employ similar ideas as in [\ref{APZ all valns on K(X)}] to provide an estimate of $IC(K(X)|K,v)$ when $(K(X)|K,v)$ is a valuation algebraic extension. Our central result is an analogue of [\ref{Dutta min pairs inertia ram deg impl const fields}, Theorem 1.3]. We further provide a natural construction of a complete sequence of key polynomials for $v$ over $K$ in the setting of valuation algebraic extensions. 
\end{abstract}


	\section{Introduction}\label{Sec intro}
	Throughout this article, we assume that $(\overline{K}(X)|K,v)$ is an extension of valued fields, where $\overline{K}$ is a fixed algebraic closure of $K$ and $X$ is an indeterminate. The extension $(K(X)|K,v)$ satisfies the celebrated Abhyankar inequality:
	\begin{equation}\label{eqn Abh inequality}
		\rr vK(X)/vK + \trdeg[K(X)v:Kv] \leq 1,
	\end{equation}
	where $vK$ and $Kv$ denote respectively the value group and residue field of $(K,v)$ and $\rr vK(X)/vK$ is the $\QQ$-dimension of the divisible hull $\QQ\tensor_{\ZZ} (vK(X)/vK)$. The above inequality is a consequence of [\ref{Bourbaki}, Chapter VI, \S 10.3, Theorem 1]. The extension $(K(X)|K,v)$ is said to be \textbf{valuation algebraic} if we have strict inequality in (\ref{eqn Abh inequality}). Otherwise, it is said to be \textbf{valuation transcendental}. Moreover, the extension is said to be \textbf{value transcendental} if we have $\rr vK(X)/vK = 1$ and \textbf{residue transcendental} if $\trdeg[K(X)v:Kv]=1$.
	
	\pars Several different objects have been introduced to study extensions of valuations to rational function fields. Minimal pairs of definition (cf. Section \ref{Sec min pairs}) are central to the investigation of valuation transcendental extensions [cf. \ref{AP sur une classe}, \ref{APZ all valns on K(X)}, \ref{APZ characterization of residual trans extns}, \ref{APZ2 minimal pairs}]. The notion of implicit constant fields (cf. Section \ref{Sec implicit const fields}) was introduced by Kuhlmann in [\ref{Kuh value groups residue fields rational fn fields}] to construct 
	valuations with prescribed value groups and residue fields. The connections between these two notions were studied thoroughly in [\ref{Dutta min fields implicit const fields}] and [\ref{Dutta min pairs inertia ram deg impl const fields}]. In particular, it has been observed in [\ref{Dutta min pairs inertia ram deg impl const fields}, Theorem 1.3] that given a minimal pair of definition $(a,\g)$ for a valuation transcendental extension $(K(X)|K,v)$ and an extension of $v$ to $\overline{K(X)}$, we have
	\[ (K(a)\sect K^r)^h \subseteq IC(K(X)|K,v) \subseteq K(a)^h,  \]
	where $K^r$ denotes the absolute ramification field of $(K,v)$ and $K^h$ denotes the henselization of $K$ (cf. Section \ref{Sec ramification theory}).  
	
	\pars Valuation algebraic extensions have been explored in detail in [\ref{APZ all valns on K(X)}]. Following an idea of MacLane [\ref{MacLane key pols}], the authors show that any valuation algebraic extension can be expressed as a limit of residue transcendental extensions. We employ similar observations to estimate the implicit constant fields of valuation algebraic extensions. When $(K(X)|K,v)$ is valuation algebraic, we observe in Section \ref{Sec valn alg extns} that there exist sequences $\{\g_\nu\}_{\nu<\l} \subseteq v(X-\overline{K})$ and $\{a_\nu\}_{\nu<\l} \subseteq (\overline{K},v)$, where $v(X-\overline{K}) := \{ v(X-a) \mid a\in\overline{K} \}$ and $\l$ is some limit ordinal, satisfying the following conditions for all $\nu<\l$: 
	\begin{itemize}
		\item $\{\g_\nu\}_{\nu<\l}$ is a cofinal well-ordered subset of $v(X-\overline{K})$,
		\item $v(X-a_\nu) = \g_\nu$,
		\item $(a_\nu,\g_\nu)$ is a minimal pair of definition for $v_{a_\nu,\g_\nu}$ over $K$,
		\item $v(X-a_\nu) \geq v(X-\s a_\nu)$ for all $K$-conjugates $\s a_\nu$ of $a_\nu$.
	\end{itemize}  
	For each $\nu<\l$, we then choose $b_\nu\in\overline{K}$ such that $K(b_\nu) = K(a_\nu) \sect K^r$. After fixing an extension of $v$ to $\overline{K(X)}$, we obtain in Theorem \ref{Thm IC for valn alg} that 
	\[ K(b_\nu \mid\nu<\l)^h \subseteq IC(K(X)|K,v) \subseteq K(a_\nu \mid \nu<\l)^h. \]
	This observation is an analogue of Theorem 1.3 of [\ref{Dutta min pairs inertia ram deg impl const fields}], which in turn is a modification of Theorem 1.1 of [\ref{Dutta min fields implicit const fields}]. 
	
	\pars Key polynomials (cf. Section \ref{Sec key pols}) are fundamental objects which play a pivotal role in the problem of describing extensions of valuations to rational function fields. They were introduced by MacLane [\ref{MacLane key pols}] and later generalized by Vaqui\'{e} [\ref{Vaquie key pols}]. An alternative form of key polynomials, also known as abstract key polynomials, was given by Spivakovsky and Novacoski in [\ref{Nova Spiva key pol pseudo convergent}]. Throughout this article, we will work with abstract key polynomials and we will simply call them key polynomials for brevity. In Theorem \ref{Thm key pols valn alg} we obtain that when $(K(X)|K,v)$ is a valuation algebraic extension, then  
	\[ \text{$\{Q_\nu\}_{\nu<\l}$ forms a complete sequence of key polynomials for $v$ over $K$,} \]
	where $Q_\nu$ is the minimal polynomial of $a_\nu$ over $K$. A special case of this result is Theorem 1.2 of [\ref{Nova Spiva key pol pseudo convergent}].    

\section{Ramification Theory} \label{Sec ramification theory}	
	A valued field $(K,v)$ is said to be \textbf{henselian} if $v$ admits a unique extension to $\overline{K}$. Every valued field has a minimal separable-algebraic extension which is henselian. This extension is unique up to valuation preserving isomorphisms over $K$. We will call this extension the \textbf{henselization} of $(K,v)$ and denote it by $K^h$. Henselization is an \textbf{immediate} extension, that is, $vK^h = vK$ and $K^hv = Kv$. An algebraic extension of henselian valued fields is again henselian. For any algebraic extension $L$ of $K$, we have that $L^h = L.K^h$. 
	
	\pars We define the distinguished group:
	\[ G^r:= \{ \s \in \Gal(\overline{K}| K) \mid v(\s a - a) > va \text{ for all } a\in K^\sep \setminus \{0\} \}, \]
where $K^\sep$ denotes the separable-algebraic closure of $K$. The corresponding fixed field in $K^\sep$ will be denoted by $K^r$ and we will call it the \textbf{absolute ramification field} of $(K,v)$. It is well-known that $K^h\subseteq K^r$ and hence $K^r$ is a henselian field. Further, the extension $(K^r|K^h,v)$ is \textbf{defectless}, that is, for any tower of field extensions $K^h \subseteq L\subseteq L^\prime\subseteq K^r$ where $L^\prime|L$ is finite, we have
\[ [L^\prime:L] = (vL^\prime:vL)[L^\prime v:Lv]. \]

\begin{Lemma}\label{Lemma imm extn}
	Assume that $(K(y)|K,v)$ is an immediate extension. Then $v(y-K)$ does not admit a maximum. 
\end{Lemma}

\begin{proof}
	Suppose that $\a:= v(y-a) = \max v(y-K)$ for some $a\in K$. The fact that $(K(y)|K,v)$ is an immediate extension implies that $v(y-K) \subset vK$. Take $b\in K$ such that $vb = \a$. Then $v \frac{y-a}{b} = 0$ and $\frac{y-a}{b} v \in K(y)v = Kv$. Consequently, there exists $c\in K$ such that $v(\frac{y-a}{b} - c) > 0$, that is, $v(y-a-bc) > vb = \a$. The fact that $a+bc \in K$ then contradicts the maximality of $\a$.
\end{proof}


\section{Minimal Pairs of Definition} \label{Sec min pairs}

Take any $a\in\overline{K}$ and $\g$ in some ordered abelian group containing $v\overline{K}$. Any polynomial $f(X) \in \overline{K}[X]$ has a unique expression of the form $f(X) = \sum_{i=0}^{n} c_i (X-a)^i$ where $c_i \in \overline{K}$. Consider the map $v_{a,\g} : \overline{K}[X] \to v\overline{K} + \ZZ\g$ given by
\[ v_{a,\g} f:= \min \{ v c_i +i\g \}, \]
and extend $v_{a,\g}$ canonically to $\overline{K}(X)$. Then $v_{a,\g}$ is a valuation transcendental extension of $v$ from $\overline{K}$ to $\overline{K}(X)$ [\ref{Kuh value groups residue fields rational fn fields}, Theorem 3.11]. Further, $(\overline{K}(X)|\overline{K}, v_{a,\g})$ is value transcendental whenever $\g\notin v\overline{K}$. Otherwise, the extension is residue transcendental.  	

\pars It has further been observed in [\ref{Kuh value groups residue fields rational fn fields}, Theorem 3.11] that $(K(X)|K,v)$ is value (residue) transcendental if and only if $(\overline{K}(X)|\overline{K}, v)$ is also value (residue) transcendental, which in turn holds if and only if $v=v_{a,\g}$ for some $a\in\overline{K}$ and $\g\in v\overline{K}(X)$. Such a pair $(a,\g) \in \overline{K}\times v\overline{K}(X)$ is said to be a \textbf{pair of definition for $v$ over $K$}. A valuation transcendental extension can have multiple pairs of definition. It has been observed in [\ref{AP sur une classe}, Proposition 3] that $(a^\prime,\g)$ is also a pair of definition for $v$ over $K$ if and only if $v(a-a^\prime)\geq \g$. 

\begin{Definition}
	Assume that $(K(X)|K,v)$ is a valuation transcendental extension. A pair $(a,\g) \in \overline{K}\times v\overline{K}(X)$ is said to be a \textbf{minimal pair of definition for $v$ over $K$} if it satisfies the following conditions:
	\sn (MP1) $(a,\g)$ is a pair of definition for $v$ over $K$,
	\n (MP2)  $v(a-a^\prime)\geq \g \Longrightarrow [K(a^\prime):K] \geq [K(a):K]$ for all $a^\prime\in\overline{K}$. 
\end{Definition}
	In other words, $a$ is of minimal degree over $K$ with the property that $(a,\g)$ is a pair of definition for $v$ over $K$. It follows from [\ref{Kuh value groups residue fields rational fn fields}, Theorem 3.11] that a valuation transcendental extension always admits a pair of definition. The well-ordering principle then implies that a valuation transcendental extension always admits a minimal pair of definition.

\begin{Lemma}\label{Lemma inclusion of value grps and res fields}
	Assume that $(K(X)|K,v)$ is a valuation transcendental extension. Take a minimal pair of definition $(a,\g)$ and a pair of definition $(a^\prime,\g)$ for $v$ over $K$. Then,
	\[ vK(a)\subseteq vK(a^\prime) \text{ and } K(a)v \subseteq K(a^\prime)v. \]
\end{Lemma}	
	
	\begin{proof}
		We first assume that $(K(X)|K,v)$ is a value transcendental extension. It follows from [\ref{Dutta min fields implicit const fields}, Remark 3.3] that 
		\[ vK(X) = vK(a)\dirsum\ZZ vQ, \]
		where $Q(X)\in K[X]$ is the minimal polynomial of $a$ over $K$. Moreover, observe that $(a^\prime,\g)$ is a minimal pair of definition for $v$ over $K(a^\prime)$. As a consequence, 
		\[ vK(a^\prime, X) = vK(a^\prime) \dirsum \ZZ\g. \]
		Take any $\a\in vK(X)$. The fact that $vK(X) \subseteq vK(a^\prime,X)$ then implies that $\a = \a^\prime + n\g$ for some $\a^\prime \in vK(a^\prime)$ and $n\in\ZZ$. Since $\g$ is not a torsion element modulo $vK$, it follows that $n=0$ and consequently, $\a = \a^\prime$. We have thus obtained that 
		\[ vK(a)\subseteq vK(a^\prime). \]
		It further follows from [\ref{Dutta min fields implicit const fields}, Remark 3.3] that 
		\[ K(a) v = K(X)v \subseteq K(a^\prime,X)v = K(a^\prime)v. \]
		
		\pars We now assume that $(K(X)|K,v)$ is a residue transcendental extension. Then $\g \in v\overline{K}$. Take the ordered abelian group $v\overline{K}\dirsum\ZZ$ endowed with the lexicographic order. Embed $v\overline{K}$ into $(v\overline{K}\dirsum\ZZ)_{\text{lex}}$ by setting 
		\[ \a \mapsto (\a,0) \text{ for all } \a\in v\overline{K}. \] 
		Set
		\[ \G:= (\g,-1). \]
		Take the extension $w:= v_{a,\G}$ of $v$ from $\overline{K}$ to $\overline{K}(X)$. Observe that,
		\[ v(a-b) \geq \G \Longleftrightarrow v(a-b) \geq \g \text{ for all } b\in\overline{K}. \]
		It follows that $(a,\G)$ is a minimal pair of definition and $(a^\prime,\G)$ is a pair of definition for $w$ over $K$. The fact that $\G\notin v\overline{K}$ implies that $(K(X)|K,w)$ is a value transcendental extension. It now follows from our preceding discussions that
		\[ vK(a)\subseteq vK(a^\prime) \text{ and } K(a)v \subseteq K(a^\prime)v. \]
	\end{proof}
	

\section{Pseudo-Cauchy Sequences}\label{Sec pCs}
 A well-ordered set $\{z_\nu\}_{\nu<\l}$ in $(K,v)$, where $\l$ is some limit ordinal, is said to form a \textbf{pseudo-Cauchy sequence (pCs)} if it satisfies the following condition:
 \[ v(z_{\nu_1} - z_{\nu_2}) < v(z_{\nu_2} - z_{\nu_3}) \text{ for all } \nu_1 < \nu_2 < \nu_3 < \l. \]
Given a polynomial $f(X)\in K[X]$ and a pCs $\{z_\nu\}_{\nu<\l}$ in $(K,v)$, the sequence $\{ vf(z_\nu) \}_{\nu<\l}$ is either ultimately constant, or ultimately monotonically increasing [\ref{Kaplansky}, Lemma 5]. We say that $\{z_\nu\}_{\nu<\l}$ is a \textbf{pCs of transcendental type} if $\{ vf(z_\nu) \}_{\nu<\l}$ is ultimately constant for all $f(X) \in K[X]$. Otherwise, we say that $\{z_\nu\}_{\nu<\l}$ is a \textbf{pCs of algebraic type}.

\pars Observe that $v (z_\nu - z_{\nu +1}) = v(z_\nu - z_\mu)$ for all $\nu<\mu<\l$. Set $\g_\nu:= v (z_\nu - z_{\nu +1})$. An element $\ell$ in some valued field extension $(L,v)$ of $(K,v)$ is said to be a \textbf{limit} of $\{z_\nu\}_{\nu<\l}$ if 
\[ v(\ell - z_\nu) = \g_\nu \text{ for all } \nu<\l. \]
If $\ell^\prime \in L$ is another limit of $\{z_\nu\}_{\nu<\l}$, then it follows from the triangle inequality that $v(\ell-\ell^\prime) \geq \g_\nu$ for all $\nu<\l$. The fact that $\{\g_\nu\}_{\nu<\l}$ is a monotonically increasing sequence without a final element then implies the following:

\begin{Lemma}\label{Lemma pCs limits}
	Assume that $\{z_\nu\}_{\nu<\l}$ is a pCs in a valued field $(K,v)$. Take an extension $(L|K,v)$ and assume that $\ell\in L$ is a limit of $\{z_\nu\}_{\nu<\l}$. Then $\ell^\prime\in L$ is also a limit of $\{z_\nu\}_{\nu<\l}$ if and only if 
	\[ v(\ell - \ell^\prime) > \g_\nu \text{ for all } \nu<\l. \]
\end{Lemma}

\begin{Definition}
	An extension $(K(X)|K,v)$ is said to be a \textbf{pure extension} if it satisfies one of the following conditions:
	\sn (PE1) $v(X-a)$ is not a torsion element modulo $vK$ for some $a\in K$, 
	\n (PE2) $v b(X-a) = 0$ and $b(X-a)v$ is transcendental over $Kv$ for some $a,b\in K$,
	\n (PE3) $X$ is a limit of some pCs of transcendental type in $(K,v)$.  
\end{Definition}
If $(K(X)|K,v)$ is a pure extension satisfying (PE3), then it follows from [\ref{Kaplansky}, Theorem 2] that $(K(X)|K,v)$ is immediate.

\section{Valuation Algebraic Extensions} \label{Sec valn alg extns}

Assume that $(K(X)|K,v)$ is a valuation algebraic extension. Then $(\overline{K}(X)|\overline{K}, v)$ is an immediate extension [\ref{Kuh value groups residue fields rational fn fields}, Lemma 3.3]. It follows from Lemma \ref{Lemma imm extn} that $v(X-\overline{K})$ does not admit a maximum. Take a cofinal well-ordered subset $\{\d_\nu\}_{\nu<\l}$ of $v(X-\overline{K})$, where $\l$ is some limit ordinal. Further, for each $\nu<\l$, take $d_\nu \in \overline{K}$ such that $v(X-d_\nu) = \d_\nu$.

\pars Take a minimal pair of definition $(a,\d_1)$ of $v_{d_1, \d_1}$ over $K$. Then $v(a - d_1) \geq \d_1$ and hence $v(X-a) \geq \d_1$. Set $\g:= v(X-a)$. By definition, $(a, \g)$ is a pair of definition for $v_{a,\g}$ over $K$. Take another pair of definition $(a^\prime, \g)$ for $v_{a,\g}$ over $K$. Then $v(a - a^\prime) \geq \g \geq \d_1$ and hence $(a^\prime, \d_1)$ is a pair of definition for $v_{d_1,\d_1}$ over $K$. The minimality of $(a,\d_1)$ then implies that $[K(a):K] \leq [K(a^\prime):K]$. It follows that $(a, \g)$ is a minimal pair of definition for $v_{a,\g}$ over $K$. If $a\notin K$, assume that $\tau a \neq a$ is a conjugate of $a$ over $K$ such that $v(X-\tau a) \geq v(X-a)$. It follows that $v(a-\tau a) \geq v(X-a) \geq \d_1$ and hence $(\tau a, \d_1)$ is also a minimal pair of definition for $v_{d_1, \d_1}$ over $K$. Replacing $a$ by $\tau a$ and employing the same arguments, we can thus further assume that $v(X-a) \geq v(X-\s a)$ for all $K$-conjugates $\s a$ of $a$. Set 
\[ a_1:= a \text{ and } \g_1 := \g. \]
Assume that we have obtained some $(a_\nu, \g_\nu)$. The fact that $\{\d_\nu\}_{\nu<\l}$ is a well-ordered cofinal subset of $v(X-\overline{K})$ implies that there exists some minimal $\mu<\l$ such that $\g_\nu < \d_\mu$. We can then construct a pair $(a^\prime, \g^\prime)$ such that $v(X-a^\prime) = \g^\prime \geq \d_\mu$, $(a^\prime,\g^\prime)$ is a minimal pair of definition for $v_{a^\prime, \g^\prime}$ over $K$ and $v(X-a^\prime) \geq v(X-\s a^\prime)$ for all $K$-conjugates $\s a^\prime$ of $a^\prime$. We set 
\[ a_{\nu+1} := a^\prime \text{ and } \g_{\nu+1}:= \g^\prime. \]
Observe that $\d_\nu \leq \g_\nu$ for all $\nu<\l$. Consequently, the fact that $\{\d_\nu\}_{\nu<\l}$ is a cofinal well-ordered subset of $v(X-\overline{K})$ implies that $\{\g_\nu\}_{\nu<\l}$ is a cofinal well-ordered subset of $v(X-\overline{K})$. We have thus constructed sequences $\{\g_\nu\}_{\nu<\l}$ and $\{ a_\nu \}_{\nu<\l}$ satisfying the following conditions for all $\nu<\l$:
\begin{itemize}
	\item $\{\g_\nu\}_{\nu<\l}$ is a cofinal well-ordered subset of $v(X-\overline{K})$,
	\item $v(X-a_\nu) = \g_\nu$,
	\item $(a_\nu,\g_\nu)$ is a minimal pair of definition for $v_{a_\nu,\g_\nu}$ over $K$,
	\item $v(X-a_\nu) \geq v(X-\s a_\nu)$ for all $K$-conjugates $\s a_\nu$ of $a_\nu$.
\end{itemize}
We fix this choice of $\{\g_\nu\}_{\nu<\l}$ and $\{a_\nu\}_{\nu<\l}$ \textit{for the rest of this article}. From [\ref{APZ all valns on K(X)}, Theorem 2.3], we obtain that 
\[ vK(a_\nu) \subseteq vK(a_\mu) \text{ and } K(a_\nu)v \subseteq K(a_\mu)v \text{ for all } \nu<\mu<\l. \]
Further, it follows from [\ref{APZ all valns on K(X)}, Theorem 5.1] that 
	\[ vK(X) = \Union\limits_{\nu<\l} vK(a_\nu) \text{ and } K(X)v = \Union\limits_{\nu<\l} K(a_\nu)v. \]
	
By definition, $\{a_\nu\}_{\nu<\l}$ is a pCs in $(K(a_\nu\mid \nu<\l), v)$ with $X$ as a limit. Suppose that $\{a_\nu\}_{\nu<\l}$ is a pCs of algebraic type. In light of [\ref{Kaplansky}, Theorem 3], there exists some $\ell \in \overline{K}$ such that $\ell$ is also a limit of $\{a_\nu\}_{\nu<\l}$. It then follows from Lemma \ref{Lemma pCs limits} that 
\[ v(X-\ell) > \g_\nu \text{ for all } \nu<\l. \]
However, this contradicts the fact that $\{\g_\nu\}_{\nu<\l}$ is a cofinal subset of $v(X-\overline{K})$. As a consequence, $\{a_\nu\}_{\nu<\l}$ is a pCs of transcendental type in $(K(a_\nu \mid \nu<\l), v)$.


\section{Implicit Constant Fields} \label{Sec implicit const fields}
 Fix an extension of $v$ to $\overline{K(X)}$. The \textbf{implicit constant field} of the extension $(K(X)|K,v)$ is defined as 
 \[ IC(K(X)|K,v) := \overline{K}\sect K(X)^h. \]
 Henselization being a separable-algebraic extension implies that $IC(K(X)|K,v)$ is a separable-algebraic extension of $K$. Moreover, the fact that $K^h \subseteq K(X)^h$ implies that $K^h \subseteq IC(K(X)|K,v)$. Consequently, $(IC(K(X)|K,v), v)$ is a henselian valued field.  

\begin{Lemma}\label{Lemma equality of value grps and res fields}
	Assume that $(K(X)|K,v)$ is a valuation transcendental extension. Take a minimal pair of definition $(a,\g)$ and a pair of definition $(a^\prime,\g)$ for $v$ over $K$. Take $b\in\overline{K}$ such that $K(b) = K(a)\sect K^r$. Then 
	\[ vK(b, a^\prime) = vK(a^\prime) \text{ and } K(b, a^\prime)v = K(a^\prime)v. \]
\end{Lemma}

\begin{proof}
	Fix an extension of $v$ to $\overline{K(X)}$. We have observed in [\ref{Dutta min pairs inertia ram deg impl const fields}, Theorem 1.3] that $K(b)^h \subseteq IC(K(X)|K,v)$. Observe that $IC(K(X)|K,v) \subseteq IC(K(a^\prime, X)|K(a^\prime), v)$ by definition. It has been observed in [\ref{Dutta min fields implicit const fields}, Lemma 5.1] that $IC(K(a^\prime, X)|K(a^\prime),v) = K(a^\prime)^h$. It then follows that $K(b)^h \subseteq K(a^\prime)^h$. In particular, $b\in K(a^\prime)^h $ and consequently, 
	\[ K(b, a^\prime)^h  = K(a^\prime)^h. \]
	The fact that the henselization is an immediate extension now yields the lemma. 
\end{proof}

We are now ready to provide an estimate of $IC(K(X)|K,v)$ when the extension $(K(X)|K,v)$ is valuation algebraic.

\begin{Theorem}\label{Thm IC for valn alg}
	Assume that $(K(X)|K,v)$ is a valuation algebraic extension. Fix an extension of $v$ to $\overline{K(X)}$. Take $\{\g_\nu\}_{\nu<\l}$ and $\{a_\nu\}_{\nu<\l}$ as in Section \ref{Sec valn alg extns}. For each $\nu<\l$, take $b_\nu\in\overline{K}$ such that 
	\[ K(b_\nu) = K(a_\nu) \sect K^r. \]
	Then,
	\begin{equation}
		K(b_\nu \mid\nu<\l)^h \subseteq IC(K(X)|K,v) \subseteq K(a_\nu \mid \nu<\l)^h.
	\end{equation}
\end{Theorem}

\begin{proof}
	We have observed in Section \ref{Sec valn alg extns} that $\{a_\nu\}_{\nu<\l}$ is a pCs of transcendental type in $(L,v)$ with limit $X$, where $L:= K(a_\nu\mid \nu<\l)$. Thus $(L(X)|L,v)$ is a pure extension by definition. It follows from [\ref{Kuh value groups residue fields rational fn fields}, Lemma 3.7] that $IC(L(X)|L,v) = L^h$. Consequently, 
	\[ IC(K(X)|K,v) \subseteq IC(L(X)|L,v) = L^h = K(a_\nu\mid \nu<\l)^h. \]
	
	\parm Take some $\nu_0<\l$. Observe that $K(b_{\nu_0}) = K(a_{\nu_0})\sect K^r$ and hence $vK(b_{\nu_0}) \subseteq vK(a_{\nu_0})$. It follows from [\ref{APZ all valns on K(X)}, Theorem 2.3] and [\ref{APZ all valns on K(X)}, Theorem 5.1] that $vK(X) = \Union\limits_{\nu<\l} vK(a_\nu)$. Consequently, 
	\[ vK(b_{\nu_0}) \subseteq vK(a_{\nu_0}) \subseteq vK(X). \]
	We now obtain from [\ref{Dutta Kuh abhyankars lemma}, Theorem 3(1)] that 
	\begin{equation}\label{eqn vK(b,X) = vK(X)}
		vK(b_{\nu_0}, X) = vK(X).
	\end{equation} 
Following the same arguments as in Section \ref{Sec valn alg extns}, we construct sequences $\{\b_\nu\}_{\nu<\l}$ and $\{c_\nu\}_{\nu<\l}$ satisfying the following conditions for all $\nu<\l$:
\begin{itemize}
	\item $\g_\nu \leq \b_\nu = v(X-c_\nu)$,
	\item $(c_\nu, \b_\nu)$ is a minimal pair of definition for $v_{c_\nu,\b_\nu}$ over $K(b_{\nu_0})$.
\end{itemize}
It follows from [\ref{APZ all valns on K(X)}, Theorem 2.3] and [\ref{APZ all valns on K(X)}, Theorem 5.1] that $K(b_{\nu_0}, c_\nu)v \subseteq K(b_{\nu_0}, c_\mu)v$ for all $\nu<\mu<\l$. Further, 
\[ K(b_{\nu_0}, X)v = \Union\limits_{\nu<\l} K(b_{\nu_0}, c_\nu)v. \]
As a consequence, 
\[ K(b_{\nu_0}, X)v = \Union\limits_{\nu_0 < \nu<\l} K(b_{\nu_0}, c_\nu)v. \]
Take some $\nu>\nu_0$. The fact that $\{\g_\mu\}_{\mu<\l}$ is cofinal in $v(X-\overline{K})$ implies that there exists some $\mu<\l$ such that
\[ \b_\nu < \g_\mu. \]
Thus $v(X-c_\nu) < v(X-a_\mu)$. It follows from the triangle inequality that $v(a_\mu-c_\nu) = \b_\nu$ and hence $(a_\mu, \b_\nu)$ is a pair of definition for $v_{c_\nu, \b_\nu}$ over $K(b_{\nu_0})$. Applying Lemma \ref{Lemma inclusion of value grps and res fields} to the extension $(K(b_{\nu_0}, X) | K(b_{\nu_0}), v_{c_\nu, \b_\nu})$, we obtain that 
\[ K(b_{\nu_0}, c_\nu)v \subseteq K(b_{\nu_0}, a_\mu)v. \]
Observe that $\g_\mu> \b_\nu\geq \g_\nu$ and hence $\mu>\nu$. Then the choice $\nu>\nu_0$ implies that $\mu>\nu_0$. As a consequence, $v(a_\mu - a_{\nu_0}) = \g_{\nu_0}$. It follows that $(a_\mu, \g_{\nu_0})$ is a pair of definition for $v_{a_{\nu_0}, \g_{\nu_0}}$ over $K$. Applying Lemma \ref{Lemma equality of value grps and res fields} to the extension $(K(X)|K, v_{a_{\nu_0}, \g_{\nu_0}})$, we obtain that
\[ K(b_{\nu_0}, a_\mu)v = K(a_\mu)v. \]  
We have thus observed that for each $\nu>\nu_0$, there exists some $\mu>\nu$ such that
\[ K(b_{\nu_0}, c_\nu)v \subseteq K(a_\mu)v. \]
As a consequence, 
\[ \Union\limits_{\nu_0 < \nu<\l}K(b_{\nu_0}, c_\nu)v \subseteq \Union\limits_{\nu_0 < \nu<\l} K(a_\nu)v.  \]
It follows from [\ref{APZ all valns on K(X)}, Theorem 2.3] and [\ref{APZ all valns on K(X)}, Theorem 5.1] that 
\[ K(X)v = \Union\limits_{\nu<\l} K(a_\nu)v = \Union\limits_{\nu_0 < \nu<\l} K(a_\nu)v. \]
The fact that $K(X)v \subseteq K(b_{\nu_0},X)v$ now yields that
\begin{equation}\label{eqn K(b,X)v = K(X)v}
	K(b_{\nu_0}, X)v = \Union\limits_{\nu_0 < \nu<\l} K(b_{\nu_0}, c_\nu)v = \Union\limits_{\nu_0 < \nu<\l} K(a_\nu)v = K(X)v.
\end{equation} 
It follows from (\ref{eqn vK(b,X) = vK(X)}) and (\ref{eqn K(b,X)v = K(X)v}) that $(K(b_{\nu_0}, X)|K(X),v)$ is an immediate extension. The fact that the henselization is an immediate extension then implies that $(K(b_{\nu_0}, X)^h | K(X)^h, v)$ is also immediate. Now, $K(b_{\nu_0})$ being contained in $K^r$ implies that $K(b_{\nu_0}, X) \subseteq K(X)^r$ [\ref{Dutta Kuh abhyankars lemma}, Theorem 3(1)]. Consequently, $(K(b_{\nu_0}, X)^h | K(X)^h, v)$ is a defectless extension. The fact that $(K(b_{\nu_0}, X)^h | K(X)^h, v)$ is immediate now implies that $K(b_{\nu_0}, X)^h = K(X)^h$, that is, $b_{\nu_0} \in K(X)^h$. As a consequence, 
\[ b_{\nu_0} \in \overline{K}\sect K(X)^h = IC(K(X)|K,v). \]
Thus $b_\nu \in IC(K(X)|K,v)$ for all $\nu<\l$. Consequently, $K(b_\nu\mid \nu<\l) \subseteq IC(K(X)|K,v)$. The fact that $IC(K(X)|K,v)$ is a henselian valued field now implies that
\[ K(b_\nu\mid \nu<\l)^h \subseteq IC(K(X)|K,v). \]
\end{proof}

A valued field $(K,v)$ is said to be a \textbf{tame field} if $(K,v)$ is henselian and $\overline{K} = K^r$. The following corollary is now immediate:

\begin{Corollary}
	Assume that $(K,v)$ is a tame valued field and $(K(X)|K,v)$ is a valuation algebraic extension. Fix an extension of $v$ to $\overline{K(X)}$. Take $\{\g_\nu\}_{\nu<\l}$ and $\{a_\nu\}_{\nu<\l}$ as in Section \ref{Sec valn alg extns}. Then,
	\[ IC(K(X)|K,v) = K(a_\nu\mid \nu<\l). \]
\end{Corollary}


\section{Key Polynomials} \label{Sec key pols}
Take a polynomial $f(X)\in K[X]$. Following [\ref{Novacoski key poly and min pairs}], we define
\[ \d(f) := \max\{ v(X-a) \mid a\in\overline{K} \text{ and } f(a)=0 \}. \]
A root $a$ of $f$ such that $\d(f) = v(X-a)$ is said to be a \textbf{maximal root} of $f$. A monic polynomial $Q(X)\in K[X]$ is said to be a \textbf{key polynomial for $v$ over $K$} if for any $f(X) \in K[X]$, we have
\[ \deg f < \deg Q \Longrightarrow \d(f) < \d(Q). \]
Take any $f(X) \in K[X]$ and a monic $Q(X) \in K[X]$. Then we have a unique expansion 
\[ f = f_0 + f_1 Q + \dotsc f_r Q^r, \]
where $f_i (X) \in K[X]$ such that $\deg f_i < \deg Q$. Consider the map $v_Q: K[X] \longrightarrow vK(X)$ given by
\[ v_Q f:= \min \{ vf_i + ivQ \}, \]
and extend it to $K(X)$ by defining $v_Q(f/g):= v_Q f - v_Q g$. A sufficient condition for $v_Q$ to be a valuation on $K(X)$ is that $Q$ be a key polynomial for $v$ over $K$ [\ref{Nova Spiva key pol pseudo convergent}, Proposition 2.6], but it is not a necessary condition [\ref{Novacoski key poly and min pairs}, Proposition 2.3].

\begin{Lemma}\label{Lemma vf leq v (a,g) f}
	Take $a\in\overline{K}$ and $f(X)\in K[X]$. Set $\g:= v(X-a)$. Then,
	\begin{align*}
		vf &> v_{a,\g}f \text{ if and only if } \d(f)> \g,\\
		vf &= v_{a,\g}f \text{ if and only if } \d(f)\leq \g.
	\end{align*}
\end{Lemma}

\begin{proof}
	Write $f(X) = z \prod_{i=1}^{n} (X-z_i)$ where $z,z_i \in\overline{K}$. By definition, $v_{a,\g} (X-z_i) = \min \{ \g, v(a - z_i) \}$. It then follows from the triangle inequality that $v_{a,\g}(X-z_i) \leq v(X-z_i)$. As a consequence, 
	\[ v_{a,\g}f \leq vf. \]
	Further, $v_{a,\g}f < vf$ if and only if $v_{a,\g} (X-z_i) < v(X-z_i)$ for some $i$, which holds if and only if $v(X-z_i) > \g = v(a-z_i)$. The lemma now follows.
\end{proof}

\begin{Lemma}\label{Lemma vQ key pol}
	Take a key polynomial $Q(X)$ for $v$ over $K$ and a maximal root $a$ of $Q$. Then,
	\[ v_{a, \d(Q)} |_{K(X)} = v_Q. \]
\end{Lemma}

\begin{proof}
	By definition, $(a,\d(Q))$ is a pair of definition for $v_{a,\d(Q)}$ over $K$. Take another pair of definition $(a^\prime, \d(Q))$. Then $v(a-a^\prime) \geq \d(Q)$ and hence $v(X-a^\prime) \geq \d(Q)$. Take the minimal polynomial $Q^\prime(X)$ of $a^\prime $ over $K$. The fact that $v(X-a^\prime) \geq \d(Q)$ implies that $\d(Q^\prime) \geq \d(Q)$. As a consequence of $Q(X)$ being a key polynomial for $v$ over $K$, we conclude that $\deg Q^\prime \geq \deg Q$, that is, $[K(a^\prime):K] \geq [K(a):K]$. It follows that $(a,\d(Q))$ is a minimal pair of definition for $v_{a,\d(Q)}$ over $K$. In light of [\ref{Novacoski key poly and min pairs}, Theorem 1.1], we conclude that $v_{a, \d(Q)} |_{K(X)} = v_Q$.
\end{proof}

\begin{Definition}
	A family $\L$ of key polynomials for $v$ over $K$ is said to form a \textbf{complete sequence of key polynomials for $v$ over $K$} if it satisfies the following conditions:
	\sn (CSKP1) $\d(Q) \neq \d(Q^\prime)$ for $Q, Q^\prime \in\L$ with $Q\neq Q^\prime$,
	\n (CSKP2) $\L$ is well-ordered with respect to the ordering given by $Q<Q^\prime$ if and only if $\d(Q) < \d(Q^\prime)$,
	\n (CSKP3) for any $f(X)\in K[X]$, there exists some $Q\in\L$ such that $vf = v_Qf$.
\end{Definition}

\begin{Theorem}\label{Thm key pols valn alg}
	Assume that $(K(X)|K,v)$ is a valuation algebraic extension. Take $\{\g_\nu\}_{\nu<\l}$ and $\{a_\nu\}_{\nu<\l}$ as in Section \ref{Sec valn alg extns}. For each $\nu<\l$, take the minimal polynomial $Q_\nu(X)$ of $a_\nu$ over $K$. Then $\{Q_\nu\}_{\nu<\l}$ forms a complete sequence of key polynomials for $v$ over $K$.
\end{Theorem}

\begin{proof}
	Observe that by our construction, $a_\nu$ is a maximal root of $Q_\nu$ for all $\nu<\l$. Consequently, 
	\[ \d(Q_\nu) = \g_\nu < \g_\mu = \d(Q_\mu) \text{ for all } \nu<\mu<\l. \]
	
	\pars Take some $\nu<\l$. Assume that $\d(g) \geq \d(Q_\nu)$ for some $g(X) \in K[X]$. Then there exists a root $a$ of $g$ such that $v(X-a) \geq \d(Q_\nu)$. Consequently,  $v(X-a) \geq v(X-a_\nu)$ and hence $v(a-a_\nu) \geq v(X-a_\nu) = \g_\nu$. It follows that $(a, \g_\nu)$ is a pair of definition for $v_{a_\nu,\g_\nu}$ over $K$. The minimality of $(a_\nu,\g_\nu)$ then implies that $[K(a):K] \geq [K(a_\nu):K]$, that is, $\deg g \geq \deg Q_\nu$. It follows that $Q_\nu(X)$ is a key polynomial for $v$ over $K$ for all $\nu<\l$.
	
	\pars We now take some $f(X) \in K[X]$ and a maximal root $a$ of $f$. The fact that $\{\g_\nu\}_{\nu<\l}$ is cofinal in $v(X-\overline{K})$ implies that there exists some $\nu<\l$ such that $\d(f) \leq \g_\nu = v(X-a_\nu) = \d(Q_\nu)$. It follows from Lemma \ref{Lemma vf leq v (a,g) f} that $vf = v_{a_\nu, \d(Q_\nu)} f$. In light of Lemma \ref{Lemma vQ key pol}, we conclude that $vf = v_{Q_\nu}f$. 
\end{proof}

\end{document}